\documentclass{amsart}

\theoremstyle{plain}
\newtheorem{Theo}{Theorem}[section]
\newtheorem{Prop}[Theo]{Proposition}
\newtheorem{Lem}[Theo]{Lemma}

\theoremstyle{definition}

\theoremstyle{remark}

\newcommand{\intl}{\int\limits}
\newcommand{\Ot}{\widetilde{\mathcal{O}}}
\begin{document}
\title{Asymptotic evaluation of $\int_0^\infty\left(\frac{\sin x}{x}\right)^n\;dx$}

\author[J.-C. Schlage-Puchta]{Jan-Christoph Schlage-Puchta}
\address{
Jan-Christoph Schlage-Puchta\\
Mathematisches Institut\\
Universit\"at Rostock\\
Ulmenstra\ss e 69, Haus 3\\
18057 Rostock, Germany}
\email{jan-christoph.schlage-puchta@uni-rostock.de}

\maketitle
\begin{abstract}
We consider the integral $\int_0^\infty\left(\frac{\sin x}{x}\right)^n\;dx$ as a function of the positive integer $n$. We show that there exists an asymptotic series in $\frac{1}{n}$ and compute the first terms of this series together with an explicit error bound.
\end{abstract}



\section{Introduction and results}
In this note we show the following.
\begin{Theo}
\label{thm:series exists}
There exists an asymptotic series
\[
\int_0^\infty\left(\frac{\sin x}{x}\right)^n\;dx \approx \frac{\pi}{2}\left(1+\sum_{\nu=1}^\infty \frac{a_\nu}{n^\nu}\right)
\]
\end{Theo}
Here the series on the right is an asymptotic series in the sense of Poincar\'e, that is, we write $f(n)\approx\sum_{\nu=0}^\infty \frac{a_\nu}{n^\nu}$, if for every fixed $k$ we have $f(n)=\sum_{\nu=0}^k \frac{a_\nu}{n^\nu} + \mathcal{O}(n^{-k-1})$. We do not know whether the series is converging for $n$ sufficiently large, but we seriously doubt it. In fact, unless there is an unexpected amount of cancelation, the coefficients $a_\nu$ grow about as fast as $\nu!$.

Our main motivation for studying this integral is the fact that the intersection of the unit cube with a plane passing through the midpoint, which is orthogonal to a diagonal of the cube has $(n-1)$-dimensional measure equal to $\frac{2\sqrt{n}}{\pi}\int_0^\infty\left(\frac{\sin t}{t}\right)^n\;dt$. For an overview of this and related results we refer the reader to the work of Chakerian and Logothetti\cite{CL}. As the sum of $n$ independent random variables uniformly distributed on $[0,1]$ equals the $L^1$-norm of a random  point in the unit cube, these intersections occur naturally in probabilistic problems. For an example we refer the reader to the work of Silberstein\cite{Silberstein}. 

For $n=1$ the integrand is the $\mathrm{sinc}$-function, which plays a crucial r\^ole in signal processing, as witnessed by the importance of the Nyquist-Shannon sampling theorem, for a historical overview we refer the reader to \cite{Luke}. After a suitable renormalisation, the function $\left(\frac{\sin x}{x}\right)^n$ is the Fourier transform of the $B$-spline of order $n$, which re also of importance in the theory of sampling. For an overview we refer the reader to the work of Butzer, Splettstößer and Stens\cite{BSS}, in particular section 4.1.

The coefficients $a_\nu$ can be computed, and the error in the approximation can be bounded effectively. As an example, we compute the following.
\begin{Prop}
\label{Prop:effective bound}
For $n\not\in\{2, 4, 6, 8, 10\}$ we have
\begin{multline*}
\int_0^\infty\left(\frac{\sin x}{x}\right)^n\;dx\\ = \sqrt{\frac{3\pi}{2n}}\left(
1-\frac{3}{20 n}-\frac{13}{1120 n^2}+\frac{27}{3200
   n^3}+\frac{52791}{3942400 n^4}+\Ot\left(\frac{7.26\cdot 10^{-3}}{n^5}\right)\right)
\end{multline*}
\end{Prop}
Here $\Ot$ denotes the effective Landau symbol, that is, the implied constant in the Landau symbol is of absolute value at most 1.

As an application we have the following, which answers a question by Schneider\cite{Schneider}.
\begin{Prop}
\label{Prop:Application}
For all $n\neq 4$ we have
\begin{equation}
\label{eq:Application}
2^{1-n}\cdot n\cdot\binom{n-1}{\lfloor(n-1)/2\rfloor}>\left(\frac{2}{\pi}\intl_0^\infty\left(\frac{\sin x}{x}\right)^n\;dx\right)^{-1}.
\end{equation}
\end{Prop}
Schneider\cite{Schneider} showed that (\ref{eq:Application}) holds for all large $n$ and stated that it probably holds for all $n\neq 4$. In particular, Proposition~\ref{Prop:Application} shows that the condition ``$n=3$ or $n$ sufficiently large'' in \cite[Theorem~2]{Schneider} can be replaced by $n\neq 4$.

All computations were performed using Mathematica 11.3.

\section{Existence of an asymptotic series}
We begin by restricting the relevant range of the integral.
\begin{Lem}
We have 
\[
0\leq \intl_1^\infty\left(\frac{\sin x}{x}\right)^n\;dx \leq e^{-n/6}
\]
\end{Lem}
\begin{proof}
For $n=1, 2$ the claim follows by direct computation, hence, from now on we assume that $n\geq 3$. In particular, the integral converges absolutely. 

The lower bound is trivial for $n$ even. For $n$ odd the integral consists of a positive integral over the range $[1, \frac{\pi}{2}]$, and a sequence of ranges of the form $[2k\pi+\frac{\pi}{2}, (2k+2)\pi+\frac{\pi}{2}]$. Considering a single interval of this form we have
\begin{multline*}
\intl_{2k\pi+\pi/2}^{2(k+1)\pi+\pi/2} \left(\frac{\sin x}{x}\right)^n\;dx =\\
 \intl_{\pi/2}^\pi\left(\frac{1}{(2k\pi+t)^n}-\frac{1}{((2k+2)\pi-t)^n} - \frac{1}{((2k+1)\pi+t)^n}+\frac{1}{(2k+3)\pi-t)^n}\right)\sin^n t, 
\end{multline*}
and the integrand is positive as $\frac{1}{t^n}$ is convex. For the upper bound we split the integral into the range $[1, \frac{\pi}{2}]$ and $[\frac{\pi}{2}, \infty]$. For the first range we have 
\[
\frac{\sin x}{x}\leq\sin 1=0.8414\dots<0.8464\dots=e^{1/6},
\]
hence, the contribution of this range is at most $(\frac{\pi}{2}-1)e^{-n/6}$. For the second part we use $\sin x\leq 1$, and see that the integral is bounded by $\int_{\pi/2}^\infty\frac{dx}{x^n} = \frac{1}{n-1}\left(\frac{\pi}{2}\right)^{n-1}$. Our claim now follows.
\end{proof}
\begin{Lem}
\label{Lem:f approx}
Let $\sum_{k\geq 1} a_k x^k$ be the Taylor series of $\log\frac{\sin x}{x}$. Then we have $|a_k|<0.517\cdot 3^{-k}$, and for $|x|<\frac{\pi}{2}$ and $K$ even the error bound
\[
\left|\log\frac{\sin x}{x} - \sum_{k=1}^K a_n x^n\right| < 1.09 \left(\frac{x}{3}\right)^{K+2}
\]
\end{Lem}
\begin{proof}
Put $f(z)=\log\frac{\sin z}{z}$. Note that $z$ is holomorphic for $|z|<\pi$ and symmetric. We have for $0<r<\pi$
\[
|a_k| = \left|\frac{f^{(k)}(0)}{n!}\right| = \frac{1}{2\pi}\left|\intl_{\partial B_r(0)}\frac{f(\zeta)}{(-\zeta)^{k+1}}\;d\zeta\right| \leq \frac{1}{2\pi r^{k+1}}\intl_{\partial B_r(0)} |f(\zeta)|\;d\zeta.
\]
We choose $r=3$ and compute the integral numerically to be $M=9.733\dots$, and obtain the claimed bound for $|a_k|$. We conclude that for $|x|\leq\frac{\pi}{2}$
\begin{multline*}
\left|\log\frac{\sin x}{x} - \sum_{k=1}^K a_k x^k\right| = \left|\sum_{k=K+2}^\infty a_k x^k\right| \leq 0.517\left(\frac{x}{3}\right)^{K+2}\sum_{k=0}^\infty \left(\frac{\pi}{6}\right)^k \leq 1.09 \left(\frac{x}{3}\right)^{K+2}
\end{multline*}
\end{proof}
\begin{Lem}
\label{Lem:f upper bound}
For $|x|\leq1$ we have
\[
\log\left(\frac{\sin x}{x}\right)\leq -\frac{x^2}{6}.
\]
\end{Lem}
\begin{proof}
We compute the first coefficients of the Taylor series as
\[
\log\left(\frac{\sin x}{x}\right) = -\frac{x^2}{6}-\frac{x^4}{180}-\frac{x^6}{2835} + \mathcal{O}(x^{8}).
\] 
By Lemma~\ref{Lem:f approx} we obtain
\[
\log\left(\frac{\sin x}{x}\right) < -\frac{x^2}{6}-\frac{x^4}{180}-\frac{x^6}{2835} + 1.09\left(\frac{x}{3}\right)^{8}.
\]
For $x\leq 1$ the second term always dominates the last one, and our claim follows.
\end{proof}

By Lemma~\ref{Lem:f upper bound} we obtain
\[
\intl_{n^{-1/3}}^1\left(\frac{\sin x}{x}\right)^n\;dx \leq \intl_{n^{-1/3}}^1\exp\left(-\frac{nx^2}{6}\right)\;dx < e^{-n^{1/3}/6}
\]
We now compute 
\begin{eqnarray*}
\intl_0^{n^{-1/3}}\left(\frac{\sin x}{x}\right)^n\;dx & = & \intl_0^{n^{-1/3}}\exp\left(n\log\left(\frac{\sin x}{x}\right)\right)\;dx\\  & = & \intl_0^{n^{-1/3}}\exp\left(-\frac{nx^2}{6} + n\sum_{k=2}^K a_{2k} x^{2k} + \mathcal{O}\left(nx^{2K+2}\right)  \right)\;dx\\
& = & \intl_0^{n^{-1/3}}\exp\left(-\frac{nx^2}{6} + n\sum_{k=2}^K a_{2k} x^{2k} \right)\;dx\\
&& + \intl_0^{n^{-1/3}}\exp\left(-\frac{nx^2}{6}\right)\left(e^{\mathcal{O}(nx^{2K+2})} -1 \right)\;dx\\
 & = & \intl_0^{\frac{n^{1/6}}{\sqrt{3}}} \exp\left(-\frac{u^2}{2} + \sum_{k=2}^K a_{2k} \frac{(3u^2)^k}{n^{k-1}} \right)\;\frac{du}{\sqrt{n/3}}\\
 && + \intl_0^{\frac{n^{1/6}}{\sqrt{3}}} e^{-\frac{u^2}{2}}\left(e^{\mathcal{O}(n^{-K}u^{2K+2})} -1 \right)\;\frac{du}{\sqrt{n/3}}.
\end{eqnarray*}
We first estimate the second integral. For $u<n^{1/6}$, the second factor of the integrand is $\mathcal{O}(n^{-\frac{2}{3}K+\frac{1}{3}})$, thus the integral is bounded by $\int_0^\infty e^{-\frac{u^2}{2}}\mathcal{O}(n^{-\frac{2}{3}K})\;du = \mathcal{O}(n^{-\frac{2}{3}K})$.

The first integral contributes to the main term. We have
\[
\exp\left(\sum_{k=2}^K a_{2k} \frac{(3u^2)^k}{n^{k-1}} \right) = 1+\sum_{k=2}^\infty\sum_{\frac{k}{2}\leq\ell\leq k-1}b_{k, \ell}^{(K)} \frac{u^{2k}}{n^\ell},
\]
where
\[
b_{k, \ell}^{(K)} = \frac{3^k}{(k-\ell)!} \underset{2\leq \kappa_i\leq K}{\sum_{\kappa_1+\dots\kappa_{k-\ell}=k}} \prod_{i=1}^{k-\ell}a_{2\kappa_i}\leq \frac{0.517^{k-\ell}}{3^k(k-\ell)!}\underset{2\leq \kappa_i\leq K}{\sum_{\kappa_1+\dots\kappa_{k-\ell}=k}} 1\leq \frac{0.517^{k-\ell}}{3^k(k-\ell)!}\binom{\ell-1}{k-\ell-1}.
\]
Here we have used the fact that
\[
\underset{2\leq \kappa_i\leq K}{\sum_{\kappa_1+\dots\kappa_{k-\ell}=k}} 1\leq\underset{2\leq \kappa_i}{\sum_{\kappa_1+\dots\kappa_{k-\ell}=k}} 1 = \binom{\ell-1}{k-\ell-1}.
\] 
To bound sums involving $b_{k, \ell}$, the following is helpful.
\begin{Lem}
\label{Lem:Fibonacci}
We have $\sum_{k/2\leq\ell\leq k-1}\binom{\ell-1}{k-\ell-1} = F_{k-1}$, where $F_n$ denotes the $n$-th Fibonacci number.
\end{Lem}
\begin{proof}
We prove our claim by induction on $k$. For $k\leq 3$ the statement is immediate.
Note that $\sum_{k/2\leq\ell\leq k-1}\binom{\ell-1}{k-\ell-1}$ equals the number of possibilities to write $k$ as a sum of integers $\kappa_i\geq 2$. Each such sum either ends with the summand 2, or it ends with a summand $>2$. The number of sums of the first kind equals the number of representations of $k-2$, as we can delete the last summand. The number of sums of the second kind equals the number of representations of $k-1$, as we can reduce the last summand by 1. Hence our claim follows.
\end{proof}
Using this notation the first integral is
\begin{multline*}
\intl_0^{\frac{n^{1/6}}{\sqrt{3}}} \exp\left(-\frac{u^2}{2} + \sum_{k=2}^K a_{2k} \frac{(3u^2)^k}{n^{k-1}} \right)\;\frac{du}{\sqrt{n/3}}\\
 = \intl_0^{\frac{n^{1/6}}{\sqrt{3}}} \exp\left(-\frac{u^2}{2}\right)\left(1+\sum_{k=2}^\infty\sum_{\frac{k}{2}\leq\ell\leq k-1}b_{k, \ell}^{(K)} \frac{u^{2k}}{n^\ell}\right)\;du\\
= \intl_0^{\frac{n^{1/6}}{\sqrt{3}}} \exp\left(-\frac{u^2}{2}\right)\left(1+\sum_{k=2}^K\sum_{\frac{k}{2}\leq\ell\leq k-1}b_{k, \ell}^{(K)} \frac{u^{2k}}{n^\ell}\right)\;du\\
 + \mathcal{O}\left(n^{-K/6}\intl_0^{\frac{n^{1/6}}{\sqrt{3}}} \exp\left(-\frac{u^2}{2}\right)\left(\sum_{k=K+1}^\infty\sum_{\frac{k}{2}\leq\ell\leq k-1}b_{k, \ell}^{(K)}\right)\;du\right).
\end{multline*}
We have
\begin{multline*}
\sum_{k=K+1}^\infty\sum_{\frac{k}{2}\leq\ell\leq k-1}b_{k, \ell}^{(K)} \leq \sum_{k=K+1}^\infty\sum_{\frac{k}{2}\leq\ell\leq k-1}\frac{0.517^{k-\ell}}{3^k(k-\ell)!}\binom{\ell-1}{k-\ell-1}\\
 \leq \sum_{k=K+1} \frac{1}{3^k} \sum_{\frac{k}{2}\leq\ell\leq k-1}\binom{\ell-1}{k-\ell-1} = \sum_{k={K+1}}^\infty\frac{F_{k-1}}{3^k}<1
\end{multline*}
As both the sum and the integral in the error term converge absolutely, the error is $\mathcal{O}(n^{-K/6})$. For the main term we interchange sum and integral, and extend the integral to $[0, \infty)$ to obtain
\begin{multline*}
\intl_0^{\frac{n^{1/6}}{\sqrt{3}}} \exp\left(-\frac{u^2}{2}\right)\left(1+\sum_{k=2}^K\sum_{\frac{k}{2}\leq\ell\leq k-1}b_{k, \ell}^{(K)} \frac{u^{2k}}{n^\ell}\right)\;du\\
 = \intl_0^{\frac{n^{1/6}}{\sqrt{3}}} \exp\left(-\frac{u^2}{2}\right)\;du + \sum_{k=2}^K\sum_{\frac{k}{2}\leq\ell\leq k-1}\frac{b_{k, \ell}^{(K)}}{n^\ell}\intl_0^{\frac{n^{1/6}}{\sqrt{3}}}\exp\left(-\frac{u^2}{2}\right)u^{2k}\;du\\
  = \intl_0^\infty \exp\left(-\frac{u^2}{2}\right)\;du + \sum_{k=2}^K\sum_{\frac{k}{2}\leq\ell\leq k-1}\frac{b_{k, \ell}^{(K)}}{n^\ell}\intl_0^\infty\exp\left(-\frac{u^2}{2}\right)u^{2k}\;du+\mathcal{O}\left(n^{\frac{2K}{3}}e^{-n^{1/3}}\right)\\
   = \sqrt{\frac{\pi}{2}}\left(1+\sum_{\ell=1}^{K-1}\frac{1}{n^\ell}\sum_{k=\ell+1}^{2\ell} b_{k,\ell}^K(2k-1)!!\right)+\mathcal{O}\left(n^{\frac{2K}{3}}e^{-n^{1/3}}\right),
\end{multline*}
where in the last step we used the fact that the moments of the normal distribution are given by
\[
\intl_0^\infty e^{-\frac{u^2}{2}}u^{2k}\;du = (2k-1)!!\sqrt{\frac{\pi}{2}},
\]
where $(2n+1)!!=1\cdot3\cdot 5\cdots (2n-1)(2n+1)$ is the double factorial, see e.g. \cite[7.4.4]{AS}.

The existence of an asymptotic series now follows.

\section{Explicite computations}

The explicite computation is similar to the asymptotic approach. However, a term of magnitude $e^{-n^{1/3}/6}$, which is negligible for large $n$, would complicate matters a lot for medium $n$. On the other hand, we can easily compute the Taylor series of $\log\left(\frac{\sin x}{x}\right)$, and the real coefficients are significantly smaller than Lemma~\ref{Lem:f approx} predicts. Therefore it is advantageous to choose different parameters. In particular we will compute higher order terms even if they are negligible for large $n$. Explicitly computing the Taylor series up to $x^{10}$ and estimating the remainder using Lemma~\ref{Lem:f approx}, we see that for $x\leq\frac{\pi}{2}$ we have
\[
0\geq \log\frac{\sin x}{x} - P(x) \geq - 1.4\cdot 10^{-12} x^{22},
\]
where $P(x)=\sum_{k=1}^{10} a_k x^{2k}$. Combining this estimate with Lemma~\ref{Lem:f upper bound} and the fact that for $\delta>0$ 
\[
\min(1, e^\delta-1) \leq \delta \max_{\delta_\in[0, \log 2} \frac{e^\delta-1}{\delta}= \frac{\delta}{\log 2}
\]
 we obtain
\begin{eqnarray*}
\intl_0^{\pi/2}\left(\frac{\sin x}{x}\right)^n\;dx & = & \intl_0^{\pi/2}\exp\left(nP(x)\right)\;dx\\
&& + \Ot\left(2.1\cdot 10^{-12} n\intl_0^{\pi/2}e^{-n\frac{x^2}{6}}x^{22}\;dx\right)\\
 & = & \intl_0^{\pi/2}\exp\left(nP(x)\right)\;dx\\ 
 &&+\Ot\left(2.1\cdot 10^{-12} 3^{23/2} n^{-23/2}\sqrt{\frac{\pi}{2}}21!!\right)\\
 & = & \intl_0^{\pi/2}\exp\left(nP(x)\right)\;dx+\Ot\left(11104 n^{-23/2}\right)
\end{eqnarray*}
We have
\[
\intl_0^{\pi/2}\exp\left(nP(x)\right)\;dx = \sqrt{\frac{3}{n}}\intl_0^{\frac{\pi\sqrt{n}}{2\sqrt{3}}} e^{-\frac{u^2}{2}}\left(1+\sum_{k\geq 2}\sum_{\frac{k}{2}\leq\ell\leq k-1}\frac{b_{k,\ell}^{(10)}}{n^\ell}u^{2k}\right)\;du
\]
Extracting the terms with $\ell\leq 10$ and extending the integral to $[0, \infty)$ we obtain
\begin{multline*}
\sqrt{\frac{3\pi}{2n}}\left(
1-\frac{3}{20 n}-\frac{13}{1120 n^2}+\frac{27}{3200
   n^3}+\frac{52791}{3942400 n^4}+\frac{482427}{66560000 n^5}\right.\\
   -\frac{124996631}{10035200000
   n^6}
   -\frac{5270328789}{136478720000 n^7}-\frac{7479063506161}{268461670400000
   n^8}\\
   \left.+\frac{6921977624613}{56518246400000 n^9} + \frac{2631854096507395099467}{1028632084480000000
   n^{10}}\right).
\end{multline*}
 We have
\[
\left(\frac{u+1}{u}\right)^{20}e^{-\frac{(u+1)^2}{2}+\frac{u^2}{2}} \leq e^{\frac{20}{u}-u- \frac{1}{2}},
\]
hence, for $u_0>5$ and $k\leq 20$ we obtain
\[
\intl_{u_0}^\infty e^{-\frac{u^2}{2}}u^k\;du\leq \frac{e^{-\frac{u_0^2}{2}}u_0^k}{1-e^{-3/2}},
\]
and we conclude that the error introduced in the extension of the integral is bounded by
\begin{multline*}
\sqrt{\frac{3}{n}}\frac{e^{-\frac{\pi n}{12}}}{1-e^{-3/2}}\left(1+\sum_{2\leq k\leq 10}\sum_{\frac{k}{2}\leq\ell\leq k-1}\frac{|b_{k,\ell}^{(10)}|}{n^\ell}\left(\frac{\pi^2n}{12}\right)^k\right)\\
\leq 2.23\cdot e^{-\frac{\pi n}{12}}\left(4.04\cdot 10^{-2}n + 8.14\cdot 10^{-4} n^2+1.1\cdot 10^{-5} n^3+1.04\cdot 10^{-7}n^4+3.69\cdot 10^{-10}n^5\right)\\
\leq 1.59\cdot 10^{-9} n^5 e^{-\frac{\pi n}{12}} < n^{-\frac{23}{2}},
\end{multline*}
provided that $n\geq 400$. We conclude that for $n\geq 400$ we have
\begin{eqnarray}
\intl_0^\infty\left(\frac{\sin x}{x}\right)^n\;dx & = & \sqrt{\frac{3\pi}{2n}}\left(
1-\frac{3}{20 n}-\frac{13}{1120 n^2}+\frac{27}{3200
   n^3}+\frac{52791}{3942400 n^4}+\Ot\left(\frac{7.25\cdot 10^{-3}}{n^5}\right)\right)\nonumber\\
\label{eq:almost there}   &&+\Ot(e^{-n/6}) + \Ot\left(\sqrt{\frac{3}{n}}\intl_0^{\frac{\pi\sqrt{n}}{2\sqrt{3}}} e^{-\frac{u^2}{2}}\sum_{k\geq 2}\sum_{\max(11, \frac{k}{2})\leq\ell\leq k-1}\frac{b_{k,\ell}^{(10)}}{n^\ell}u^{2k}\;du\right)
\end{eqnarray}
Next we bound the contribution of a summand with $\ell$ large. As $e^{-\frac{u^2}{2}}u^{2k}$ increases for $u<\sqrt{2k}$ and decreases for $u>2k$, we have the bound
\[
\intl_0^{\frac{\pi\sqrt{n}}{2\sqrt{3}}} e^{-\frac{u^2}{2}}u^{2k}\;du\leq\begin{cases}
(2k-1)!!, & k\leq \frac{\pi^2 n}{24},\\
\left(\frac{\pi\sqrt{n}}{2\sqrt{3}}\right)^{2k+1} e^{-\frac{\pi^2 n}{24}}, & k>\frac{\pi^2 n}{24}.
\end{cases}
\]
Hence, the contribution of the range $k>\frac{\pi^2 n}{24}$ is at most
\begin{multline*}
\sum_{k>\frac{\pi^2 n}{24}} \left(\frac{\pi\sqrt{n}}{2\sqrt{3}}\right)^{2k+1} e^{-\frac{\pi^2 n}{24}}\sum_{k/2\leq\ell\leq k-1} \frac{|b_{k, \ell}|}{n^\ell}\\
 \leq \sqrt{n}e^{-\frac{\pi^2 n}{24}}\sum_{k>\frac{\pi^2 n}{24}}\sum_{k/2\leq\ell\leq k-1} \frac{0.907^{2k+1} (0.517n)^{k-\ell}}{3^k(k-\ell)!}\binom{\ell-1}{k-\ell-1}\\
 \leq \sqrt{n}e^{-\frac{\pi^2 n}{24}+0.517 n}\sum_{k>\frac{\pi^2 n}{24}}\frac{0.907^{2k+1}F_{k-1}}{3^k} < 1.64\sqrt{n}\cdot 0.796^n.
\end{multline*}
The contribution of the summands belonging to a single $k\leq\frac{\pi^2 n}{24}$ is at most
\begin{multline*}
(2k-1)!!\sum_{\frac{k}{2}\leq\ell\leq k-1} \frac{b_{k,\ell}}{n^\ell} \leq \frac{(2k)!}{2^k k!}\sum_{\frac{k}{2}\leq\ell\leq k-1}\frac{ 0.517^{k-\ell}}{3^k(k-\ell)! n^\ell}\binom{\ell-1}{k-\ell-1}\\
\leq \frac{F_{k-1}(2k)!}{6^k k!}\sum_{\frac{k}{2}\leq\ell\leq k-1}\frac{ 0.517^{k-\ell}}{(k-\ell)! n^\ell} \leq 2\frac{F_{k-1}(2k)!}{6^k k!}\cdot\frac{0.517^{k/2}}{(k/2)!n^{k/2}} < 0.479^k\left(\frac{k}{n}\right)^{k/2}.
\end{multline*}
Hence, the sum over all $k\geq K$ is bounded by $1.45\cdot 0.479^K\left(\frac{K}{n}\right)^{K/2}$. We take $K=35$.

It remains to bound the range $\ell\geq 11$, $k\leq 34$. Here we obtain by direct computation 
\[
\sum_{k\leq 34} \sum_{\max(11,\frac{k}{2})\leq\ell\leq k-1}\frac{b_{k, \ell}}{n^\ell}\leq \frac{5.5\cdot 10^7}{n^{11}},
\]
provided that $n>400$. We see that the last two error terms in (\ref{eq:almost there}) are bounded by
\[
e^{-n/6} + 1.45\cdot 0.479^{35}\left(\frac{35}{n}\right)^{\frac{35}{2}} + \frac{5.5\cdot 10^7}{n^{11}},
\]
which for $n\geq 400$ is bounded by $\frac{10^{-5}}{n^5}$, hence our claim follows in this range. 

Finally for $n\leq 400$ we check Proposition~\ref{Prop:effective bound} directly using the following result due to
Chakerian and Logothetti \cite{CL}.
\begin{Lem}
We have
\[
\intl_0^\infty\left(\frac{\sin x}{x}\right)^n\;dx = \frac{\pi}{2^n(n-1)!}\sum_{j=0}^{\lfloor n/2\rfloor} (-1)^j\binom{n}{j}(n-2j)^{n-1}.
\]
\end{Lem}

\section{Proof of Proposition~\ref{Prop:Application}}

Note first that for all $n\geq 1$ we have that the $n^{-3}$ term in the asymptotic series in Proposition~\ref{Prop:effective bound} dominates the error term, hence, for all $n$ the right hand side of (\ref{eq:Application}) is at most  
\[
\sqrt{\frac{\pi n}{6}}\left(1-\frac{3}{20n}-\frac{13}{1120n^2}\right)^{-1}.
\]
To estimate the left hand side we use Stirling's formula in the following form, see \cite[6.1.38]{AS}
\begin{Lem}
We have
\[
n! = \left(\frac{n}{e}\right)^n\sqrt{2\pi n} e^{\Ot\left(\frac{1}{12 n}\right)}.
\]
\end{Lem}
We conclude that
\[
\binom{2m}{m} = \frac{2m!}{m!^2} = \frac{2^{2m}}{\sqrt{\pi m}}e^{\Ot\left(\frac{1}{6m}\right)}
\]
and
\[
\binom{2m+1}{m}=\frac{1}{2}\binom{2m+2}{m+1} = \frac{2^{2m+1}}{\sqrt{\pi (m+1)}}e^{\Ot\left(\frac{1}{6(m+1)}\right)}.
\]
We obtain that the left hand side of (\ref{eq:Application}) is 
\[
\frac{n}{\sqrt{\pi\lfloor n/2\rfloor}}e^{\Ot\left(\frac{1}{3n}\right)}\geq \sqrt{\frac{2n}{\pi}}\left(1-\frac{1}{3n}\right)
\]
Hence, to prove Proposition~\ref{Prop:Application} it suffices to check that Proposition~\ref{Prop:effective bound} holds, and that 
\[
\sqrt{\frac{2n}{\pi}}\left(1-\frac{1}{3n}\right) > \sqrt{\frac{\pi n}{6}}\left(1-\frac{3}{20n}-\frac{13}{1120n^2}\right)^{-1},
\]
that is, 
\[
\left(1-\frac{1}{3n}\right)\left(1-\frac{3}{20n}-\frac{13}{1120n^2}\right) > \sqrt{\frac{\pi^2}{12}} = 0.9068\ldots,
\]
and we see that this inequality holds for $n\geq 6$. We conclude that Proposition~\ref{Prop:Application} holds for $n\not\in\{1,2,3,4,5,6,8,10\}$, and by direct inspection we find that it holds for all $n\neq 4$.

\end{document}